\documentclass[12pt]{article}

\def\UseBibLatex{1}

\makeatletter
\def\input@path{{styles/}}
\makeatother

\newcommand{\UsePackage}[1]{%
  \IfFileExists{styles/#1.sty}{%
      \usepackage{styles/#1}%
   }{%
      \IfFileExists{../styles/#1.sty}{%
         \usepackage{../styles/#1}%
      }{%
         \usepackage{#1}%
      }%
   }%
}

\usepackage[T1,T5]{fontenc}
\usepackage{lmodern}
\usepackage{textcomp}

\usepackage{amsmath}%
\usepackage{amssymb}%
\usepackage[table]{xcolor}%

\usepackage{todonotes}
\usepackage[in]{fullpage}%

\usepackage{titlesec}%
\titlelabel{\thetitle. }%
\usepackage{xcolor}%
\usepackage{mleftright}%
\usepackage{xspace}%
\usepackage{graphicx}
\usepackage{hyperref}%
\usepackage{multirow}
\usepackage{amsmath,amssymb}
\usepackage{indentfirst}
\usepackage{amsmath}
\usepackage[column=O]{cellspace}
\usepackage{tabularx}
\usepackage{amsthm}

\usepackage{hyperref}%
\hypersetup{%
      unicode,
      breaklinks,%
      colorlinks=true,%
      urlcolor=[rgb]{0.25,0.0,0.0},%
      linkcolor=[rgb]{0.5,0.0,0.0},%
      citecolor=[rgb]{0,0.2,0.445},%
      filecolor=[rgb]{0,0,0.4},
      anchorcolor=[rgb]={0.0,0.1,0.2}%
}
\usepackage[ocgcolorlinks]{ocgx2}

\theoremstyle{plain}%
\newtheorem{theorem}{Theorem}[section]
\newtheorem{lemma}[theorem]{Lemma}

\newtheorem{prop}[theorem]{Proposition}
\newtheorem{cor}[theorem]{Corollary}

\theoremstyle{definition}
\newtheorem{definition}[theorem]{Definition}
\newtheorem{example}[theorem]{Example}

\theoremstyle{remark}
\newtheorem{remark}[theorem]{Remark}
\newtheorem{rmk}[theorem]{Remark}
\newcommand{\ben}{\begin{enumerate}}
\newcommand{\een}{\end{enumerate}}
\newcommand{\R}{\mathbb{R}}

\newcommand{\Z}{\mathbb{Z}}

\providecommand{\emphind}[1]{}%
\renewcommand{\emphind}[1]{\emph{#1}\index{#1}}

\definecolor{blue25emph}{rgb}{0, 0, 11}

\providecommand{\emphic}[2]{}
\renewcommand{\emphic}[2]{\textcolor{blue25emph}{%
      \textbf{\emph{#1}}}\index{#2}}

\providecommand{\emphi}[1]{}%
\renewcommand{\emphi}[1]{\emphic{#1}{#1}}

\definecolor{almostblack}{rgb}{0, 0, 0.3}

\providecommand{\emphw}[1]{}%
\renewcommand{\emphw}[1]{{\textcolor{almostblack}{\emph{#1}}}}%

\providecommand{\emphOnly}[1]{}%
\renewcommand{\emphOnly}[1]{\emph{\textcolor{blue25}{\textbf{#1}}}}

\newcommand{\LucThanks}[1]{%
   \thanks{%
      Department of Mathematics; %
      University of Pittsburgh; %
      200 S. Craig St.; %
      Pittsburgh, PA 15260, USA; %
      \href{mailto:ldt37@pitt.edu}{ldt37@pitt.edu}; %
      \url{https://luc-ta.github.io/}. %
   #1%
   }%
}

\newcommand{\PeytonThanks}[1]{%
   \thanks{%
      Department of Mathematics; %
      University of California, Davis; %
      One Shields Ave; %
      Davis, CA, 95616, USA; %
      \href{mailto:ppwood@ucdavis.edu}{ppwood@ucdavis.edu}; %
      \url{https://sites.google.com/view/peytonpwood/}. %
   #1%
   }%
}

\newcommand{\HLink}[2]{\hyperref[#2]{#1~\ref*{#2}}}
\newcommand{\HLinkSuffix}[3]{\hyperref[#2]{#1\ref*{#2}{#3}}}

\providecommand{\deflab}[1]{}
\renewcommand{\deflab}[1]{\label{def:#1}}

\providecommand{\eqlab}[1]{}%
\renewcommand{\eqlab}[1]{\label{equation:#1}}

\newcommand{\remove}[1]{}%

\usepackage[inline]{enumitem}

\newlist{compactenumA}{enumerate}{5}%
\setlist[compactenumA]{topsep=0pt,itemsep=-1ex,partopsep=1ex,parsep=1ex,%
   label=(\Alph*)}%

\newlist{compactenuma}{enumerate}{5}%
\setlist[compactenuma]{topsep=0pt,itemsep=-1ex,partopsep=1ex,parsep=1ex,%
   label=(\alph*)}%

\newlist{compactenumI}{enumerate}{5}%
\setlist[compactenumI]{topsep=0pt,itemsep=-1ex,partopsep=1ex,parsep=1ex,%
   label=(\Roman*)}%

\newlist{compactenumi}{enumerate}{5}%
\setlist[compactenumi]{topsep=0pt,itemsep=-1ex,partopsep=1ex,parsep=1ex,%
   label=(\roman*)}%

\newlist{compactitem}{itemize}{5}%
\setlist[compactitem]{topsep=0pt,itemsep=-1ex,partopsep=1ex,parsep=1ex,%
   label=\ensuremath{\bullet}}%

\providecommand{\BibLatexMode}[1]{}
\providecommand{\BibTexMode}[1]{}

\ifx\UseBibLatex\undefined%
  \renewcommand{\BibLatexMode}[1]{}
  \renewcommand{\BibTexMode}[1]{#1}
\else
  \renewcommand{\BibLatexMode}[1]{#1}
  \renewcommand{\BibTexMode}[1]{}
\fi

\BibLatexMode{%
   \usepackage[bibencoding=utf8,style=alphabetic,backend=biber]{biblatex}%
   \UsePackage{my_biblatex}%
}

\numberwithin{equation}{section}%

\newcommand{\ul}{u_l}
\newcommand{\vl}{v_l}
\newcommand{\ur}{u_r}
\newcommand{\vr}{v_r}
\newcommand{\dl}{d_l}
\newcommand{\dr}{d_r}

\DeclareMathOperator{\tb}{tb}
\DeclareMathOperator{\rot}{rot}
\DeclareMathOperator{\Aut}{Aut}
\DeclareMathOperator{\Inn}{Inn}
\renewcommand{\phi}{\varphi}
\newcommand{\inv}{^{-1}}
\newcommand{\tr}{\triangleright}
\DeclareMathOperator{\id}{id}
\DeclareMathOperator{\Conj}{Conj}

\usepackage{mathtools}

\BibLatexMode{%
   \bibliography{template}
}

\title{Distinguishing Power of 4-Legendrian Permutation Racks}

\author{%
   L\d\uhorn c Ta
   \LucThanks{}%
   \and%
   Peyton Phinehas Wood%
   \PeytonThanks{}%
}

\date{\today}

\begin{document}
\maketitle

\begin{abstract}
    We study 4-Legendrian racks and their effectiveness at distinguishing Legendrian knots. We prove that permutation racks with 4-Legendrian rack structures cannot distinguish Legendrian knots that share the same knot type, Thurston–Bennequin number, and rotation number. However, they also recover these three classical invariants.
\end{abstract}



\section{Introduction}
    In 2023 and 2025, Kimura \cite{kimura-jktr} and Cheng and He \cite{Cheng_2025fundamentalgeneralizedlegendrianrack} posed the question of whether certain algebraic objects called \emph{4-Legendrian racks} can distinguish Legendrian knots with the same topological knot type, Thurston--Bennequin number, and rotation number. In this work, we provide a negative answer for all 4-Legendrian permutation racks.
        \begin{theorem}[Theorem \ref{perm4LRcannotdist}]\label{perm4LRcannotdist0}
    Permutation racks with 4-Legendrian structures cannot distinguish isotopy classes of Legendrian knots with the same classical invariants.
\end{theorem}

Since permutation racks are the main examples of racks that are not quandles in the literature, Theorem \ref{perm4LRcannotdist0} and a similar result of Kimura for 4-Legendrian quandles \cite*[Thm.\ 4.2.3]{kimura-thesis} settle this question in the negative for very large classes of 4-Legendrian racks.

On the other hand, we also show that the Legendrian knot invariants provided by 4-Legendrian racks are at least as powerful as the classical invariants.
    \begin{theorem}[Corollary \ref{cor:classical}]\label{thm:classical}
        The fundamental 4-Legendrian rack $F(L)$ recovers the classical invariants of $L$.
    \end{theorem}
    Theorem \ref{thm:classical} can be viewed as a stronger version of the main theorem in a paper of Cheng and He \cite{Cheng_2025fundamentalgeneralizedlegendrianrack} from 2025. To prove their result, Cheng and He computed a normal form for words in certain colorings of Legendrian knots by \emph{GL-racks}, which are a special class of 4-Legendrian racks. By adapting this method to 4-Legendrian racks more generally, we obtain a key lemma (Lemma \ref{lem:permutation}) used to prove both Theorem \ref{perm4LRcannotdist0} and Theorem \ref{thm:classical}.
    
    Kimura introduced 4-Legendrian racks in his 2024 PhD thesis \cite{kimura-thesis} by equipping algebraic structures called \emph{racks} with extra structure. Racks are generalizations of algebraic structures called \emph{quandles}, which Joyce \cite{Joyce_1982} and Matveev \cite{1984_Matveev} independently introduced in the 1980s to construct invariants of knots and links. In 1992, Fenn and Rourke \cite{Fenn_Rouke_racks92} introduced racks to study framed knots and links. It was quickly realized that every semi-framed non-split link embedded in a 3-manifold has a fundamental rack that classifies both the link and the 3-manifold.

Although rack theory and Legendrian knot theory \cite{etnyre} were popularized concurrently in the 1990s, it was not until 2017 when Kulkarni and Prathamesh \cite{KP_2017rackinvariantslegendrianknots} introduced the first rack-theoretic invariants of Legendrian knots. They called these invariants \emph{$n$-Legendrian racks}, and they used these invariants to distinguish infinitely many Legendrian unknots \cite[Main Thm.\ 2]{KP_2017rackinvariantslegendrianknots}.\footnote{Despite the name, 4-Legendrian racks are not a special class of $n$-Legendrian racks; in fact, the former vastly generalizes the latter.} In 2021, Ceniceros, Elhamdadi, and Nelson \cite{CEN_2021legendrianrackinvariantslegendrian} generalized $n$-Legendrian racks by introducing \emph{Legendrian racks}, which are 4-Legendrian racks where $\ul=\ur=\dl=\dr$; cf.\ \eqref{eq:d-inv}. They used Legendrian racks to distinguish certain Legendrian trefoils and connected sums of Legendrian trefoils \cite[Sec.\ 5]{CEN_2021legendrianrackinvariantslegendrian}.

Further generalizing Legendrian racks, Kimura \cite{kimura-jktr} and Karmakar, Saraf, and Singh \cite{Karmakar_2025} independently introduced \emph{GL-racks} (also called \emph{generalized Legendrian racks} or \emph{bi-Legendrian racks}) in 2023. GL-racks are 4-Legendrian racks where $\ul=\ur$ (and, hence, $\dl=\dr$; cf.\ \eqref{eq:d-inv}). They distinguish infinitely many Legendrian trefoils and even more Legendrian unknots than $n$-Legendrian racks can; see \cite[Thm.\ 4.1]{kimura-jktr} and \cite[Thms.\ 4.7 and 4.8]{Karmakar_2025}. 
In 2025, Cheng and He \cite[Thm.\ 1.1]{Cheng_2025fundamentalgeneralizedlegendrianrack} further showed that GL-racks distinguish Legendrian knots at least up to the absolute values of their classical invariants.

However, none of the above examples of Legendrian knots answer the question of whether rack invariants can distinguish Legendrian knots with the same topological knot type and classical invariants. Kimura \cite{kimura-jktr} and Cheng and He \cite{Cheng_2025fundamentalgeneralizedlegendrianrack} both posed this question. In his PhD thesis, Kimura \cite[Sec.\ 4.2]{kimura-thesis} discovered several examples of nonequivalent Legendrian knots sharing these invariants that \emph{cannot} be distinguished using 4-Legendrian racks. He also showed that 4-Legendrian quandles cannot distinguish \emph{any} such Legendrian knots \cite[Thm.\ 4.2.3]{kimura-thesis}. Our main theorem (Theorem \ref{perm4LRcannotdist0}) extends this result to all 4-Legendrian permutation racks. We note that while Kimura's result uses a theorem from contact topology, our approach is purely algebro-combinatorial.

On the algebraic side of things, the first author \cite{Ta_2025classificationstructuregeneralizedlegendrian} studied GL-racks from group-theoretic, categorical, and universal-algebraic perspectives in 2025. In particular, a simplified but equivalent definition of GL-racks led to a group-theoretic characterization of GL-structures, answering a question posed by Karmakar, Saraf, and Singh in a previous version of \cite{Karmakar_2025}. Combining \cite[Thm.\ 5.6]{Ta_2025classificationstructuregeneralizedlegendrian} and \cite[Thm.\ 10.1]{Ta_2025goodinvolutionsconjugationsubquandles} yields the surprising result that the categories of racks, Legendrian racks, and GL-quandles are isomorphic. Moreover, certain involutory GL-racks have connections to \emph{symmetric racks} (also called racks with \emph{good involutions}) \cite*[Thm.\ 10.1]{Ta_2025goodinvolutionsconjugationsubquandles}, which are used to distinguish surface-knots in $\R^4$. These connections formalized an observation that Karmakar, Saraf, and Singh made in a previous version of \cite{Karmakar_2025}.

    This paper is organized as follows. In Section \ref{review}, we review common rack theory ideas and definitions and provide several examples of families of racks that we explore later in this paper. 
    In Section \ref{4LR}, we define 4-Legendrian racks and discuss their classification as algebraic structures. 
    In Section \ref{LK}, we review classical and rack-theoretic invariants of Legendrian knots.
    In Section \ref{main}, we prove Theorems \ref{perm4LRcannotdist0} and \ref{thm:classical}. In Section \ref{future}, we propose areas for future research.


\section{Review of Racks}\label{review}

\begin{definition}[Rack, \cite{Elhamdadi_Nelson_2015}]
    A {\em rack} $(X,\triangleright)$ is a set X with binary operation $\triangleright \colon X \times X \to X$ satisfying the following two axioms:
    \begin{enumerate}
        \item R1 (Invertibility): For all y $\in$ X, the map $\beta_y \colon X \to X$ defined by $\beta_y(x)=x\triangleright y$ is invertible. We denote $\beta^{-1}_y(x)$ by $x \triangleright^{-1} y$. 
        \item R2 (Self-distributivity): For all $x, y, z \in X$, we have $(x \triangleright y) \triangleright z = (x \triangleright z) \triangleright (y \triangleright z) $.
    \end{enumerate}
    \end{definition}

    A rack is called a {\em quandle} if the binary operation $\triangleright$ is idempotent, that is, $x \triangleright x = x$ for all $x\in X$. We say that $|X|$ is the \textit{order} of $(X,\tr)$. Finite racks have been enumerated up to isomorphism for orders up to 13 and fully classified for orders up to 11 \cite{Vor_Yang_2019}.


In the following, let $n\geq 1$ be a positive integer.

\begin{example}[Trivial quandle]
    The \emph{trivial quandle} of order $n$, denoted $T_n$, has underlying set $X=\mathbb{Z}/n\Z$ and quandle operation $x \triangleright y \coloneq x$.
\end{example}

\begin{example}[Dihedral quandle]
    The \emph{dihedral quandle} of order $n$, denoted $R_n$, has underlying set $X=\mathbb{Z}/n\Z $ and quandle operation $x \triangleright y \coloneq 2y-x \pmod{n}$. 
\end{example}

 \begin{example}[Alexander quandle]
     Pick $t \in \Z/n\Z$ such that $\gcd(t,n)=1$. The quandle with underlying set $X=\mathbb{Z}/n\Z$ and quandle operation $x \triangleright y \coloneq (1-t)y+tx$ is called an \emph{Alexander quandle} of order $n$.
 \end{example}

 \begin{example}[Conjugation quandle]
     Let $G$ be any group. The \emph{conjugation quandle} is the quandle $\Conj(G)\coloneq(G,\triangleright)$ with operation $a \triangleright b \coloneq bab^{-1}$. 
 \end{example}

 \begin{example}[Core quandle]
     Let $G$ be any group. The \emph{core quandle} is the quandle $\operatorname{Core}(G)\coloneq(G,\triangleright)$ with operation $a \triangleright b \coloneq ba^{-1}b$. 
 \end{example}

 \begin{example}[Takasaki quandle]
     Let $A$ be any additive abelian group. The \emph{Takasaki quandle} is the quandle $T(A)\coloneq(A,\tr)$ with operation $a \triangleright b \coloneq 2b - a $.
 \end{example}

 \begin{remark}
     When $A$ is abelian, Core($A$) = $T(A)$. Additionally, choosing $A=\Z/n\Z$ for Takasaki quandles yields the dihedral quandle, $R_n$.
 \end{remark}

     All of the above examples are quandles. 
     Examples of racks that are not quandles are much less common in the literature; the following is the usual example.

\begin{example}[Permutation rack]
   Let $X$ be a set, and let $\sigma \in S_X$ be a permutation of $X$. The \emph{permutation rack} or \emph{constant action rack} is the rack $X_\sigma\coloneq(X,\tr)$ with operation $x \triangleright y \coloneq \sigma(x)$.
\end{example}

\begin{remark}
    Letting $\sigma \coloneq \id_X$ results in the trivial quandle. Otherwise, $X_\sigma$ is a rack that is not a quandle as it will violate idempotency for at least one element.
\end{remark}

\begin{example}[$(t,s)$-rack]
    Let $X$ be a module over the quotient ring $\mathbb{Z}[t^{\pm},s] / (s^2-s(1-t))$. The rack with underlying set $X$ and rack operation $x \triangleright y \coloneq tx+sy$ is called a \emph{$(t,s)$-rack}.
\end{example}

\begin{remark}
    All $(t, 1-t)$-racks are Alexander quandles. When $s \ne 1-t$, the resulting rack is not a quandle as $x \tr x \ne x$.
\end{remark}

\subsection{Rack automorphisms}

\begin{definition}
    Let $(X,\tr_X)$ and $(Y,\tr_Y)$ be racks. A map $\phi\colon X\to Y$ is called a \emph{rack homomorphism} if $\phi(x_1\tr_X x_2)=\phi(x_1)\tr_Y\phi(x_2)$ for all $x_1,x_2\in X$. 
    
    If $(X,\tr_X)=(Y,\tr_Y)$ and $\phi$ is bijective, we say that $\phi$ is a \emph{rack isomorphism}. A rack isomorphism from $(X,\tr_X)$ to itself is a \emph{rack automorphism}. The \emph{automorphism group} of $(X,\tr_X)$ is denoted by $\Aut(X)$.
\end{definition}

Important to the theory of racks is the following canonical automorphism $\pi$. The name comes from the map's association with kinks in diagrams of framed knots.

\begin{definition}[{\cite[p.\ 149]{Elhamdadi_Nelson_2015}}]
    Let $(X,\tr)$ be a rack. The \emph{kink map} $\pi\colon X\to X$ is the function defined by $x\mapsto x\tr x$.
\end{definition}

One can show that $\pi$ is a rack automorphism; the inverse map $\pi\inv$ is given by $x\mapsto x\tr\inv x$. One can also show that $\pi$ and $\pi\inv$ commute with all other rack homomorphisms, implying that $\pi$ and $\pi\inv$ lie in the center of the category of racks. See, for example, \cite[Sec.~2.2]{Ta_2025classificationstructuregeneralizedlegendrian}.

\begin{example}
    A rack $(X,\tr)$ is a quandle if and only if $\pi=\id_X$. Thus, $\pi$ can be thought of as the obstruction to being a quandle.
\end{example}

\begin{example}\label{pi_is_sigma}
    If $X_\sigma$ is a permutation rack, then $\pi=\sigma$.
\end{example}

Another important class of rack automorphisms is the following.

\begin{definition}
    Let $(X,\tr)$ be a rack. The \emph{inner automorphism group} $\Inn(X)$ (also called the \emph{operator group} or \emph{right multiplication group}) is the subgroup of the symmetric group $S_X$ generated by the maps $\beta_x$ ranging over all $x\in X$:
    \[
    \Inn(X)\coloneq\langle\beta_x\mid x\in X\rangle.
    \]
\end{definition}

\begin{remark}\label{rmk:normal}
    The rack axioms state precisely that $\Inn(X)$ is a subgroup of $\Aut(X)$. A quick check shows that this subgroup is normal.
\end{remark}


\section{4-Legendrian Racks}\label{4LR}

In the following, let $(X,\tr)$ be a rack.

\begin{definition}[{Cf.\ \cite[Sec.\ 4.1]{Ta_2025classificationstructuregeneralizedlegendrian}}]
    Define the group $U_X$ to be the centralizer \[U_X:=C_{\Aut (X)}(\Inn(X)).\] We say that elements of $U_X$ are \emph{GL-structures} on $(X,\tr)$.\ Ordered pairs of GL-structures $(\ul,\ur)\in U_X\times U_X$ are called \emph{4-Legendrian structures} on $(X,\tr)$.
\end{definition}

Note that, by Remark \ref{rmk:normal}, $U_X$ is a normal subgroup of $\Aut (X)$.

\begin{definition}
    A \emph{4-Legendrian rack} is a quadruple $(X,\tr,\ul,\ur)$ where $(X,\tr)$ is a rack and $(\ul,\ur)$ is a 4-Legendrian structure on $(X,\tr)$.
\end{definition}

\begin{definition}
    Let $(X,\tr_X,\ul,\ur)$ and $(Y,\tr,\vl,\vr)$ be 4-Legendrian racks. A rack homomorphism (resp.\ isomorphism) $\phi\colon X\to Y$ is called a \emph{4-Legendrian rack homomorphism} (resp.\ \emph{isomorphism}) if $\phi\ul=\vl\phi$ and $\phi\ur=\vr\phi$.
\end{definition}

\begin{example}[Cf.\ {\cite[Ex.\ 3.7]{kimura-jktr}}]\label{ex:permutations}
    Let $X_\sigma$ be a permutation rack. Then a 4-Legendrian structure is precisely a pair $(\ul,\ur)$ of permutations of $X$ that commute with $\sigma$.
\end{example}

\begin{example}[Cf.\ {\cite[Ex.\ 3.6]{kimura-jktr}}]
    Let $G$ be a group, and let $g,h\in Z(G)$. Let $\ul,\ur\colon G\to G$ be the multiplication maps $x\mapsto gx$ and $x\mapsto hx$, respectively. Then $(\Conj(G),\ul,\ur)$ is a 4-Legendrian quandle.
\end{example}

\begin{rmk}
    For all racks $(X,\tr)$, the maps $\id_X$, $\pi$, and $\pi\inv$ are GL-structures, though they are not necessarily all distinct. Thus, every rack can be equipped with at least one 4-Legendrian structure; a similar statement holds for GL-racks but not for Legendrian racks.
\end{rmk}

We comment on the relationship between our definition of 4-Legendrian racks and the original definition. In his PhD thesis introducing 4-Legendrian racks, Kimura \cite[Sec.\ 4.2]{kimura-thesis} defined a 4-Legendrian rack as a sextuple $(X,\tr,\ul,\ur,\dl,\dr)$ in which $(X,\tr)$ is a rack and $\ul,\ur,\dl,\dr\colon X\to X$ are functions that satisfy the following eight axioms for all $x,y\in X$:
    \begin{equation*}
        \begin{aligned}
            \dl\ur = \ur\dl&=\dr\ul=\ul\dr,\\
        \dr\ul(x\tr x)&= x,\\
        \dl(x\tr y)&= \dl(x)\tr y,\\
        \ul(x\tr y)&= \ul(x)\tr y,\\
        \dr(x\tr y)&= \dr(x)\tr y,\\
        \ur(x\tr y)&= \ur(x)\tr y,\\
        x\tr \dl(y)&=x\tr y=x\tr\dr(y),\\
        x\tr \ul(y)&=x\tr y=x\tr\ur(y).
        \end{aligned}
    \end{equation*}
    Morphisms in this category are defined in the obvious way.
    
    With some work, one can show that Kimura's definition of 4-Legendrian racks is equivalent to our definition. Namely, $\dl$ and $\dr$ are respectively determined entirely by $\ur$ and $\ul$ as well as~$\tr$:
    \begin{equation}\label{eq:d-inv}
        \dl=\ur\inv\pi\inv,\qquad \dr=\ul\inv\pi\inv.
    \end{equation}
    In fact, the forgetful functor defined by $(X,\tr,\ul,\ur,\dl,\dr)\mapsto(X,\tr,\ul,\ur)$ is an isomorphism of categories. 
    This fact is proven similarly to the analogous result for GL-racks  \cite[Prop.\ 3.12]{Ta_2025classificationstructuregeneralizedlegendrian}; we omit the details here.

    In light of this equivalence, we use our definition to study algebraic properties of 4-Legendrian racks, and we use Kimura’s definition to study Legendrian knot invariants.

\begin{example}\label{ex:permutation2}
    Given a 4-Legendrian permutation rack as in Example \ref{ex:permutations}, combining \eqref{eq:d-inv} with Example \ref{pi_is_sigma} yields
    \[
    (\ur\dl)\inv=\sigma=(\ul\dr)\inv.
    \]
\end{example}

\begin{rmk}
    A GL-rack is precisely a 4-Legendrian rack in which $\ul=\ur$. A Legendrian rack is precisely a GL-rack in which $\dl=\ur$ and (as a result) $\dr=\ul$.
\end{rmk}

\subsection{Classification of 4-Legendrian Racks}

Let $(\ul,\ur)$ and $(\vl,\vr)$ be 4-Legendrian structures on $(X,\tr)$. By definition, the 4-Legendrian racks $(X,\tr,\ul,\ur)$ and $(X,\tr,\vl,\vr)$ are isomorphic if and only if there exists a rack automorphism $\phi\in\Aut(X)$ such that $\vl=\phi \ul\phi\inv$ and $\vr=\phi \ur\phi\inv$. Equivalently, the pairs $(\ul,\ur)$ and $(\vl,\vr)$ are \emph{simultaneously conjugate} in $\Aut(X)$; that is, the 4-Legendrian structures lie in the same orbit of $U_X\times U_X$ under the diagonal conjugation action of $\Aut (X)$. This action exists because $U_X$ is a normal subgroup of $\Aut(X)$. To summarize, we have just observed the following.

\begin{prop}[{Cf.\ \cite[Thm.\ 4.1]{Ta_2025classificationstructuregeneralizedlegendrian}}]\label{prop:isoms}
    The isomorphism classes of 4-Legendrian structures on $(X,\tr)$ are precisely the orbits of $U_X\times U_X$ under the diagonal conjugation action of $\Aut (X)$.
\end{prop}

\begin{example}
    Let $n\geq 0$ be a nonnegative integer, and let $T_n$ be the trivial quandle of order $n$. Then $U_X=\Aut (T_n)=S_n$, so isomorphism classes of $4$-Legendrian racks with underlying rack $T_n$ correspond to orbits of $S_n\times S_n$ under the diagonal action of $S_n$ by conjugation. These orbits are counted in OEIS sequence A110143 \cite{oeis}; see \cite{mathoverflow-conjugacy}. (We verified this example for all $n\leq 6$ using the computer search described below.) 
\end{example}

\begin{rmk}
    Even if we have GL-rack isomorphisms $(X,\tr,\ul)\cong (X,\tr,\vl)$ and $(X,\tr,\ur)\cong (X,\tr,\vr)$, the 4-Legendrian racks $(X,\tr,\ul,\ur)$ and $(X,\tr,\vl,\vr)$ are not necessarily isomorphic. For example, let $T_3$ be the trivial quandle of order $3$, and let $\ul=\ur=\vl=(2,3)$ and $\vr=(1,3)$ in cycle notation.
\end{rmk}

\begin{rmk}
    Given a 4-Legendrian rack $(X,\tr,\ul,\ur)$, it is not necessarily isomorphic to $(X,\tr,\ur,\ul)$. For example, let $(X,\tr)$ be any rack such that $U_X$ contains a nonidentity element $\phi\neq\id_X$, and let $(\ul,\ur)\coloneq (\phi,\id_X)$.
\end{rmk}

Using Proposition \ref{prop:isoms}, we implemented a \texttt{GAP} \cite{GAP4} program that classifies all $4$-Legendrian racks of a given order $n\leq 11$ up to isomorphism. This program uses the classification of racks up to order $11$ from \cite{Vor_Yang_2019}, which is where the $n\leq 11$ bound comes from.

We were able to complete the search for all $n\leq 6$; our code and data is available in a GitHub repository \cite{github-code}.
Table \ref{tab:isoms} enumerates our data. The table also enumerates 4-Legendrian involutory racks and 4-Legendrian kei; this is motivated by the connections between involutory GL-racks and symmetric racks shown in \cite[Sec.\ 10]{Ta_2025goodinvolutionsconjugationsubquandles}.

\begin{table}[h]\centering

\begin{tabular}{l|ccccccc}
Order            & 0 & 1 & 2 & 3  & 4   & 5    & 6     \\ \hline
Racks            & 1 & 1 & 8 & 33 & 249 & 1592 & 15944 \\
Involutory racks & 1 & 1 & 8 & 24 & 196 & 850  & 9248  \\
Quandles         & 1 & 1 & 4 & 16 & 84  & 448  & 3137  \\
Kei              & 1 & 1 & 4 & 16 & 74  & 342  & 2228 
\end{tabular}
\caption{Number of $4$-Legendrian structures on various families of racks of order up to $6$, counted up to isomorphism.}
\label{tab:isoms}
\end{table}

\begin{cor}
Let $G\coloneq\Aut(X)$. 
    If $\Inn (X)\leq Z(G)$, then the isomorphism classes of 4-Legendrian structures on $(X,\tr)$ are precisely the orbits of $G\times G$ under the diagonal conjugation action of $G$. In particular, if $G$ is abelian, then these isomorphism classes are precisely the group $G\times G$.
\end{cor}

\begin{proof}
    This follows from the fact that $\Inn(X)\leq Z(G)$ if and only if $U_X=G$.
\end{proof}

\begin{example}
    Let $n\in\Z^+$ be a positive integer, let $\sigma\in S_n$ be an $n$-cycle, and let $X_\sigma$ be the corresponding permutation rack of order $n$. Then $\Aut(X_\sigma)=\langle\sigma\rangle\cong\Z/n\Z$, which is abelian, so the set of isomorphism classes of 4-Legendrian structures on $R$ is $\langle\sigma\rangle\times\langle\sigma\rangle$. In particular, there are exactly $n^2$ isomorphism classes of 4-Legendrian racks with underlying rack $X_\sigma$. (Our \texttt{GAP} search verified this fact for all $n\leq 6$.)
\end{example}

\begin{example}
    Let $F$ be the free rack on one generator, identified as the permutation rack $\Z_\sigma$ defined by $\sigma(k):=k+1$ for all $k\in\Z$. Then $\Aut (F)=\langle\sigma\rangle\cong\Z$, which is abelian. Hence, the set of isomorphism classes of 4-Legendrian structures on $F$ is $\{(\sigma^m,\sigma^n)\mid m,n\in\Z\}\cong\Z^2$.
\end{example}

\section{4-Legendrian Rack Invariants}\label{LK}

\subsection{Legendrian Knots}

Legendrian knots are important objects of study in contact topology. 
We briefly review several concepts from the theory. For a formal treatment, we refer the reader to the survey of Etnyre \cite{etnyre}.

The \emph{standard contact structure} is the kernel of the differential 1-form $dz-y\, dx$ in $\R^3$, which is depicted as a plane field in Figure \ref{fig:ContactStructure}. A smooth knot in $\R^3$ is called \emph{Legendrian} if it lies everywhere tangent to the standard contact structure. 

\begin{figure}
    \centering
    \includegraphics[alt={An assignment of a plane to each point in Euclidean three-space. The planes are horizontal along the x-axis and twist in the y-direction. They approach verticality, but never reach it.},scale=0.3]{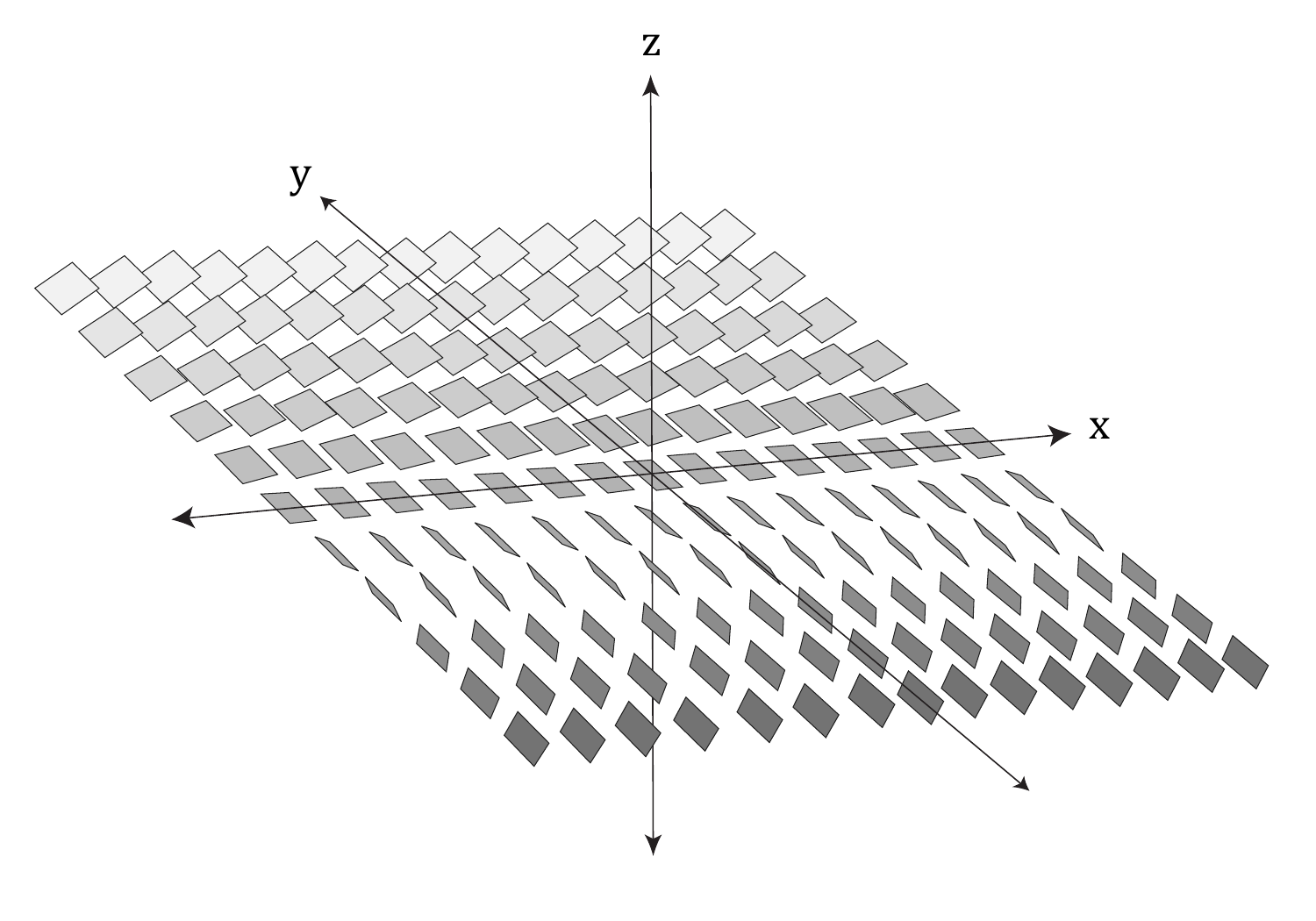}
    \caption{The standard contact structure on $\mathbb{R}^3$, depicted as an assignment of a plane to each point.}
    \label{fig:ContactStructure}
\end{figure}

Legendrian knots are usually studied via their projections to the $xz$-plane viewed from the negative $y$-direction. These are called \emph{front projections}. 
Front projections of Legendrian knots have two key features that distinguish them from projections of smooth knots. First, since tangent lines can never be vertical, front projections have cusps in place of vertical tangencies. Second, due to the direction of twisting, the overstrand at each crossing is always the strand having the more negative slope. Note that an \emph{oriented} front projection never has two leftward-oriented or rightward-oriented cusps placed adjacent to each other. For example, Figure \ref{fig:trefoil0} (cf.\ Figure \ref{fig:trefoil}) depicts a Legendrian left-handed trefoil in its front projection.

\begin{figure}
    \centering
    \includegraphics[alt={An oriented front projection of an Legendrian left-handed trefoil.},width=0.4\linewidth]{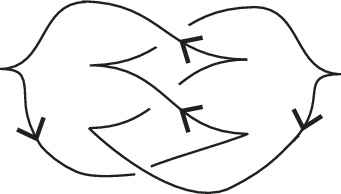}
    \caption{Oriented front projection of a Legendrian left-handed trefoil.}
    \label{fig:trefoil0}
\end{figure}

A major problem in contact topology is the classification of Legendrian knots up to Legendrian isotopy. Every smooth knot type has infinitely many Legendrian representatives. Many of them can be distinguished using their topological knot type and their two \emph{classical invariants}, which we define below. Given a front projection of a Legendrian knot $L$, let $w(L)$ denote the \emph{writhe} of the front projection, and let $U$ and $D$ denote the numbers of upward-oriented and downward-oriented cusps. Note that $U$ and $D$ must be positive, and $U+D$ must be even.

\begin{definition}[{\cite[Sec.\ 2.6.2]{etnyre}}]
    The \emph{Thurston--Bennequin number} of $L$ is the integer
    \[
    \tb(L)\coloneq w(L)-\frac{1}{2}(D+U).
    \]
    The \emph{rotation number} of $L$ is the integer
    \[
    \rot(L)\coloneq \frac{1}{2}(D-U).
    \]
    These two integers are called the \emph{classical invariants} of $L$.
\end{definition}

\begin{example}
    The classical invariants of the Legendrian left-handed trefoil in Figure \ref{fig:trefoil0} are $(\tb,\rot)=(-6,-1)$.
\end{example}

However, there exist many nonequivalent Legendrian knots that share the same topological knot type and classical invariants; see, for example, \cite[Sec.\ 4]{etnyre}. This motivates the search for Legendrian knot invariants that distinguish such pairs of Legendrian knots.

\begin{definition}
    A Legendrian knot invariant is called \emph{effective} if it distinguishes some pair of Legendrian knots sharing the same topological knot type and classical invariants.
\end{definition}

In particular, Kimura \cite{kimura-jktr} and Cheng and He \cite{Cheng_2025fundamentalgeneralizedlegendrianrack} posed the open questions of whether GL-racks and, more generally, 4-Legendrian racks provide effective invariants of Legendrian knots. While this question remains open in general, Kimura \cite[Thm.\ 4.2.3]{kimura-thesis} gave a negative answer for all 4-Legendrian quandles, and our main theorem (Theorem \ref{perm4LRcannotdist}) provides a negative answer for all 4-Legendrian permutation racks.

\subsection{Coloring Invariants}

In 2024, Kimura \cite[Sec.\ 4.2]{kimura-thesis} introduced an invariant of Legendrian knots $L$ called the \emph{fundamental 4-Legendrian rack}, which we denote by $F(L)$. This is defined in the same vein as the fundamental GL-rack of $L$, which Karmakar, Saraf, and Singh introduced in \cite[Thm.\ 4.3]{Karmakar_2025}, and the fundamental quandle of a smooth knot, which was independently introduced by Joyce \cite{Joyce_1982} and Matveev \cite{1984_Matveev}.

We outline the construction here; see \cite[Sec.\ 4.2]{kimura-thesis} for details, and cf.\ \cite{Karmakar_2025,Cheng_2025fundamentalgeneralizedlegendrianrack}. View 4-Legendrian racks as an algebraic theory with two binary operations $\tr^{\pm 1}$ and four unary operations $\ul^{\pm 1},\ur^{\pm 1}$. By general results from universal algebra, \emph{free 4-Legendrian racks} exist and are characterized up to isomorphism by the usual universal property. Moreover, we can take quotients of 4-Legendrian racks by congruence relations.

Fix an oriented front projection of a Legendrian knot $L$. Label the \emph{arcs} (i.e., connected components) of the front projection by $x_1,\dots,x_n$, and let $F\coloneq \langle x_1,\dots,x_n\rangle$ be the free 4-Legendrian rack generated by $x_1,\dots,x_n$. Starting at any crossing in the front projection, traverse the knot using its given orientation. At each cusp, impose a label $\ul,\ur,\dl$, or $\dr$ depending on the orientation of the cusp in the obvious way; cf.\ \eqref{eq:d-inv}. At each crossing, impose a relation on $F$ between arcs as illustrated in Figure \ref{fig:crossings}. Define $F(L)$ to be the quotient of $F$ by the congruence relation generated by these $n$ crossing relations.

\begin{figure}
    \centering
    \includegraphics[alt={Two local pictures of crossings in front projections of Legendrian knots. The first is a negative crossing, and the second is a positive crossing. The arcs are labeled using the relations of the fundamental 4-Legendrian rack.},width=0.4\linewidth]{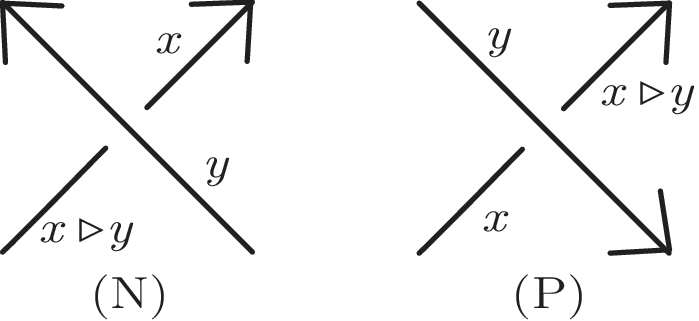}
    \caption{Relations imposed on $F(L)$ between arcs at negative and positive crossings.}\label{fig:crossings}
\end{figure}

One can show that the isomorphism class of $F(L)$ is an invariant of $L$ by using Legendrian Reidemeister moves; see \cite[Sec.\ 4.2]{kimura-thesis}. Of course, it is difficult to distinguish infinite 4-Legendrian racks directly from their generators and relations, so we instead consider \emph{colorings} of $F(L)$ by finite 4-Legendrian racks.
These are easier to compute since $F(L)$ is finitely presented (having $n$ generators and $n$ relations in the above construction).

\begin{definition}[{\cite[Def.\ 23]{kimura-thesis}}]
    Let $R\coloneq(X,\tr,\ul,\ur)$ be a 4-Legendrian rack. A \emph{coloring} of $L$ by $R$ is a 4-Legendrian rack homomorphism $F(L)\to R$.
\end{definition}

Since $F(L)$ is an invariant of $L$, the set $\operatorname{Hom}(F(L),X)$ of colorings of $L$ by a given 4-Legendrian rack $(X,\tr,\ul,\ur)$ is also an invariant of $L$; cf.\ \cite[Rem.\ 30]{kimura-thesis}.

\begin{example}
    Let $L$ be the Legendrian left-handed trefoil with classical invariants $(\tb,\rot)=(-6,-1)$ in Figure \ref{fig:trefoil0}. By following the procedure described above, we obtain Figure \ref{fig:trefoil}. Thus, the fundamental 4-Legendrian rack $F(L)$ is the quotient of the free 4-Legendrian rack $\langle x_1,x_2,x_3\rangle$ by the congruence relation generated by the relations
    \[
    x_3\tr x_1=\dl\ur(x_2),\quad  x_2\tr x_3=\ul\dr(x_1), \quad x_1\tr x_2=\ul\ur(x_3).
    \]
\end{example}

\begin{figure}
    \centering
    \includegraphics[alt={An oriented front projection of an Legendrian left-handed trefoil. The arcs are labeled in a way that corresponds to the fundamental 4-Legendrian rack.},width=0.5\linewidth]{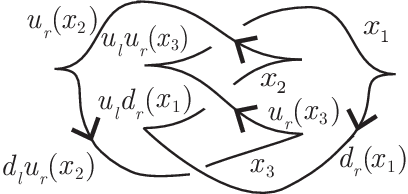}
    \caption{Legendrian left-handed trefoil with arcs labeled.}\label{fig:trefoil}
\end{figure}

\section{Proof of the Main Theorem}\label{main}

Fix an oriented front projection of a Legendrian knot $L$ with $U$ upward-oriented cusps and $D$ downward-oriented cusps, and let $(X_\sigma,\ul,\ur,\dl,\dr)$ be a permutation rack equipped with a 4-Legendrian structure. To prove our main result, we construct a canonical form for relations in the image of a given coloring of $L$ by $X_\sigma$.

This idea is similar to an approach Cheng and He took to prove the main theorem of \cite{Cheng_2025fundamentalgeneralizedlegendrianrack}. Namely, they observed that every relation in the canonical presentation of the fundamental GL-rack of $L$ can be written in the form
\begin{align*}
    x&=u^{p_n}d^{q_n}(x_n)\tr^{\epsilon_n}x_{k_n}\\
    &= u^{p_n+p_{n-1}}d^{q_n+q_{n-1}}(x_{n-1}\tr^{\epsilon_{n-1}}x_{k_n-1})\tr^{\epsilon_n}x_{k_n}\\
    &= \dots\\
    &= u^U d^D(\dots(x\tr^{\epsilon_1}x_{k_1})\tr^{\epsilon_2}\dots)\tr^{\epsilon_n}x_{k_n},
\end{align*}
where $U=\sum p_n$, $D=\sum q_n$, and $\epsilon_n\in\{\pm 1\}$ with $\sum\epsilon_n=w(L)$.

The obstruction to taking this exact approach with 4-Legendrian racks is that $\ul$ and $\dr$ generally do not commute with $\ur$ and $\dl$. Nevertheless, we have the following.

\begin{prop}\label{prop:word}
    Let $x$ be one of the generators of $F(L)$, and consider either of the two relations of $F(L)$ corresponding to a crossing in which $x$ is an understrand. Then this relation can be written in the form
    \begin{equation}\label{eq:rln}
        x=W (\dots(x\tr^{\epsilon_1}x_{k_1})\tr^{\epsilon_2}\dots)\tr^{\epsilon_n}x_{k_n},
    \end{equation}
    where $W$ is a word of length $U+D$ consisting of $U$ letters in $\{\ul,\ur\}$ and $D$ letters in $\{\dl,\dr\}$ (without inverses), and $\epsilon_n\in\{\pm 1\}$ with $\sum\epsilon_n=w(L)$. Moreover, $W$ alternates between letters in $\{\ul,\dl\}$ and letters in $\{\ur,\dr\}$.
\end{prop}

\begin{proof}
    In the given relation for $x$, it follows from \eqref{eq:d-inv} that we can move all of the cusp functions $\ul,\ur,\dl,\dr$ to the leftmost part of the word, giving a string $W$ of length $U+D$. 
    
    To obtain the defining relations for $F(L)$, we had to traverse the front projection of $L$ using its given orientation. It follows that none of the relations obtained in this way contain the inverse of any of the cusp functions. Moreover, it is impossible for two rightward-oriented cusps or two leftward-oriented cusps to appear in a row. It follows that our word alternates between left and right cusp functions. Hence, $W$ satisfies the desired properties. The rest is identical to the above calculation of Cheng and He.
\end{proof}

\begin{lemma}\label{lem:permutation}
    In the image of any coloring of $L$ by a 4-Legendrian permutation rack $X_\sigma$, any word of the form \eqref{eq:rln} can be rewritten in either the form
    \[
    x=(\dl\dr)^{\rot(L)}\sigma^{\rot(L)+\tb(L)}(x)\qquad \text{or}\qquad x=(\dr\dl)^{\rot(L)}\sigma^{\rot(L)+\tb(L)}(x).
    \]
\end{lemma}

\begin{proof}
    Since $X_\sigma$ is a permutation rack, the string to the right of $W$ in \eqref{eq:rln} collapses to $\sigma^{w(L)}$; that is, $x=W\sigma^{w(L)}$. 
    
    Since $L$ is a Legendrian knot, $U+D$ must be even, and $U,D>0$. It follows from Proposition \ref{prop:word} that one of $u_ld_r$, $d_ru_l$, $u_rd_l$, or $d_lu_r$ appears in $W$. Introduce a $\sigma\sigma^{-1}$ in the middle of this pair. Since $\sigma$ must commute with $u_l, u_r, d_l, \text{ and } d_r$, we can commute the $\sigma\inv$ to the rightmost part of the word. By Example \ref{ex:permutation2}, we can then use the $\sigma$ to cancel out the $u_ld_r$, $d_ru_l$, $u_rd_l$, or $d_lu_r$. We can repeat this action until there are no longer any upward cusp functions $\ul,\ur$ in our word.

    The word now only contains $D-U=2\rot(L)$ total $d_l$'s and $d_r$'s. Either $d_l$ or $d_r$ begin the word, and these two functions must alternate. This leaves us with $\rot(L)$ pairs of $d_ld_r$'s or $d_rd_l$'s followed by $w(L) - U=\rot(L)+\tb(L)$ total $\sigma$'s in our word, as desired.
\end{proof}

As an aside, Lemma \ref{lem:permutation} allows us to read $\rot(L)$ and $\tb(L)$ off of the relations of a coloring $F(L)\to X_\sigma$. This yields the following, which may be seen as a stronger version of the main theorem of \cite{Cheng_2025fundamentalgeneralizedlegendrianrack} for 4-Legendrian racks.

\begin{cor}[Theorem \ref{thm:classical}]\label{cor:classical}
    The fundamental 4-Legendrian rack $F(L)$ recovers the classical invariants of $L$.
\end{cor}

\begin{rmk}
    Some authors consider the topological knot type to be a classical invariant of Legendrian knots. The statement of Corollary \ref{cor:classical} still holds under this convention. To see this, recall that the fundamental quandle is a complete invariant of smooth knots (see \cite{1984_Matveev,Joyce_1982}). Given a Legendrian embedding $L$ of a knot $K$, we can recover the fundamental quandle of $K$ from $F(L)$ by imposing the congruence relation generated by $x\sim\ul(x)\sim\dr(x)\sim\ur(x)\sim\dl(x)$ for all $x\in F(L)$; cf.\ \cite[Rem.\ 4.6]{Karmakar_2025}.
\end{rmk}

We now prove the main theorem.

\begin{theorem}[Theorem \ref{perm4LRcannotdist0}]\label{perm4LRcannotdist}
    4-Legendrian permutation racks cannot distinguish isotopy classes of Legendrian knots with the same classical invariants.
\end{theorem}

\begin{proof}
    Let $(X_\sigma, u_l, u_r, d_l, d_r)$ be any 4-Legendrian permutation rack. Assume $L_1$ and $L_2$ are two distinct Legendrian isotopy classes of knots with the same classical invariants. Fix oriented front projections of $L_1$ and $L_2$.

    In the image of a coloring of $L_1$ or $L_2$ by $(X_\sigma, u_l, u_r, d_l, d_r)$, we have $x \triangleright y = \sigma(x)$ at each crossing relation. Therefore, we can combine all the relations for $L_1$ and $L_2$ into a single generator and single relation $\langle x: x=R_{L_1}(x)\rangle$ and $\langle x: x=R_{L_2}(x)\rangle$, where $x$ can start at any arc of the front projections of $L_1$ and $L_2$, respectively.

    By Lemma \ref{lem:permutation}, both of the relations $R_{L_i}(x)$ can be written in either the form 
    \[
    x=(\dl\dr)^{\rot(L)}\sigma^{\rot(L)+\tb(L)}(x)\qquad \text{or}\qquad x=(\dr\dl)^{\rot(L)}\sigma^{\rot(L)+\tb(L)}(x).
    \]
    We can choose $x$ in each respective front projection such that the word made of $d_l$'s and $d_r$'s begins with the same function. Since $L_1$ and $L_2$ share the same classical invariants, it follows that these two relations are identical. Hence, by the definition of a 4-Legendrian rack coloring, the sets of colorings of $L_1$ and $L_2$ by $(X_\sigma, u_l, u_r, d_l, d_r)$ are equal.
\end{proof}


\section{Areas for Further Inquiry}\label{future}

We conclude by proposing areas for further research. These complement the open question of whether there exist 4-Legendrian racks that can distinguish Legendrian knots with the same classical invariants.

\begin{itemize}
    \item Study the coloring invariant of 4-Legendrian racks that are not permutation racks.
    \item Extend the algebraic results for GL-racks in \cite{Ta_2025classificationstructuregeneralizedlegendrian} to 4-Legendrian racks.
    \item Develop enhancements of the 4-Legendrian rack coloring invariant. 
    \item Study the algebraic properties and distinguishing power of \emph{4-Legendrian biracks}, which Kimura introduced in his PhD thesis \cite{kimura-thesis}.
    \item Develop rack-theoretic invariants of Legendrian graphs.
    \item Develop rack-theoretic invariants of transverse knots.
\end{itemize}

\section*{Acknowledgments}

We would like to thank the organizers of the Unknot V conference for giving the authors a chance to meet each other and Jose Ceniceros, whose work with Mohamed Elhamdadi and Sam Nelson we are building off of. The second author would also like to thank Lenny Ng, Roger Casals, Orsola Capovilla-Searle and Tye Lidman for fruitful conversations concerning Legendrian knots.



\BibLatexMode{\printbibliography}

\end{document}